\documentclass{article}

\usepackage{amsthm}
\usepackage{amsmath, amssymb}
\usepackage{hyperref}
\usepackage{tikz}
\usetikzlibrary{matrix}
\usepackage[matrix, arrow, curve]{xy}

\newtheorem{Thm}{Theorem}[section]

\newtheorem{Lem}[Thm]{Lemma}

\newtheorem{Rem}[Thm]{Remark}

\title{Stably free modules of rank $2$ over certain real smooth affine threefolds}

\author{Tariq Syed\\
Mathematisches Institut\\
Heinrich-Heine-Universit{\"a}t D{\"u}sseldorf\\
Universit{\"a}tsstra{\ss}e 1\\
40225 D{\"u}sseldorf, Germany\\
tariq.syed@gmx.de}

\date{\today} 

\begin{document}

\maketitle

\begin{abstract}
Let $R$ be a real smooth affine domain of dimension $3$ such that $R$ has either no real maximal ideals or the intersection of all real maximal ideals in $R$ has height at least $1$. Then we prove that all stably free $R$-modules of rank $2$ are free if and only if the Hermitian $K$-theory group $W_{SL}(R)$ is trivial.
\end{abstract}

\tableofcontents

\section{Introduction}

A finitely generated module $P$ over a commutative ring $R$ is called stably free if there exists an isomorphism of the form $P \oplus R^m \cong R^n$ for some integers $m,n \geq 0$. A fundamental question in algebraic $K$-theory and commutative algebra is the question under which conditions stably free modules are actually free. If $R$ is a Noetherian commutative ring of Krull dimension $d$, then classical results imply that a stably free $R$-module of rank $\geq d+1$ is always free (cf. \cite[Chapter IV, Theorem 3.4]{HB}). Improvements of this classical result were proven for affine algebras over algebraically closed fields: If $R$ is affine algebra of dimension $d$ over an algebraically closed field $k$, then any stably free $R$-module of rank $d$ is free (cf. \cite{S1}); furthermore, if $R$ is a smooth affine algebra of dimension $d \geq 3$ over an algebraically closed field $k$ with $(d-1)! \in k^{\times}$, then stably free $R$-modules of rank $d-1$ are free (cf. \cite[Theorem 7.5]{FRS}).\\
For affine algebras over $\mathbb{R}$, analogues of the results over algebraically closed fields mentioned above do not exist in general. For example, if we let $R = \mathbb{R}[x,y,z]/\langle x^2 + y^2 + z^2 -1 \rangle$ be the real algebraic $2$-sphere, then it is well-known that the kernel of the $R$-linear homomorphism $(x,y,z): R^3 \rightarrow R, (\lambda_{1},\lambda_{2},\lambda_{3}) \mapsto \lambda_{1} x + \lambda_{2} y + \lambda_{3} z$ is a stably free module of rank $2$ which is not free. As shown in S. Banerjee's beautiful work (cf. \cite{B}), the situation for affine $\mathbb{R}$-algebras changes substantially when one focuses on specific $\mathbb{R}$-algebras:

\begin{Thm}[{{\cite[Theorem 2.8]{B}}}]\label{T1.1}
Let $R$ be a commutative ring satisfying the following condition \textbf{P}: The ring $R$ is an affine $\mathbb{R}$-algebra of dimension $d$ such that
\begin{itemize}
\item it has no real maximal ideals or
\item the intersection of all real maximal ideals in $R$ has height at least $\geq 1$.
\end{itemize}
Then all stably free $R$-modules of rank $d$ are free.
\end{Thm}

The purpose of this article is prove a result on stably free modules of rank $d-1$ over regular integral domains of dimension $d$ satisfying condition \textbf{P} in Theorem \ref{T1.1} in the first interesting case, namely in dimension $3$ (in dimensions $d=1$ and $d=2$, stably free modules of rank $d-1$ are easily seen to be free). We prove the following cohomological criterion for all stably free modules of rank $2$ over a regular integral domain of dimension $3$ satisfying condition \textbf{P} to be free (cf. Theorem \ref{T3.4}):

\begin{Thm}\label{T1.2}
Let $R$ be a regular integral domain of dimension $d=3$ satisfying condition $\textbf{P}$ from Theorem \ref{T1.1}. Then all stably free $R$-modules of rank $2$ are free if and only if $W_{SL}(R) = 0$.
\end{Thm}

The abelian group $W_{SL}(R)$ is a Hermitian $K$-theory group and was introduced in \cite[\S 3]{SV}. Theorem \ref{T1.2} is proven by analyzing the generalized Vaserstein symbol modulo SL introduced in \cite[Section 3.A]{Sy1}; as a matter of fact, we prove that this map is bijective in Theorem \ref{T3.3} and hence induces a group structure on the set of isomorphism classes of oriented stably free $R$-modules of rank $2$ (Remark \ref{R3.5}). The bijectivity of the Vaserstein symbol is essentially a consequence of a transitivity statement on the actions of certain symplectic groups on the set of unimodular rows of length $4$ over $R$, which is proven in Theorem \ref{T3.1}.\\
The paper is structured as follows: In Section \ref{2} we give a brief introduction unimodular rows, the groups $W_E (R)$ and $W_{SL}(R)$ of a commutative ring $R$ as needed for this paper and recall some statements proven in \cite{B} and \cite{Sy2} which will be used in the proof of the main results of this paper. Then we prove the main results of this paper in Section \ref{3}.

\section{Preliminaries}\label{Preliminaries}\label{2}

Let $R$ be a commutative ring with unit and $n \geq 1$ an integer. A unimodular row of length $n$ over $R$ is a row vector $(v_{1},...,v_{n})$ of length $n$ with $v_{i} \in R$, $1 \leq i \leq n$, such that $\langle v_{1},...,v_{n} \rangle = R$. We denote the set of unimodular rows of length $n$ over $R$ by $Um_{n}(R)$. The group $GL_{n}(R)$ of invertible $n \times n$-matrices acts on the right on $Um_{n}(R)$ by matrix multiplication. The same holds automatically for any subgroup of $GL_{n}(R)$; in particular, the subgroup $SL_{n}(R)$ of matrices with determinant $1$, the subgroup $E_{n}(R)$ generated by elementary matrices act on the right on $Um_{n}(R)$. If $\chi$ is an alternating invertible matrix of rank $2n$, then the subgroup $Sp(\chi)$ of $GL_{2n}(R)$ consisting of matrices which are symplectic with respect to $\chi$ act on the right on $Um_{2n}(R)$. For integers $n,i$ with $1 \leq i \leq n$, we denote by $\pi_{i,n}$ the unimodular row $(0,...,0,1,0,...,0)$ of length $n$ over $R$ with $1$ in the $i$th slot and $0$'s elsewhere and $\pi_{n}=\pi_{n,n}$.\\
Similarly, a unimodular column of length $n$ over $R$ is a column vector ${(v_{1},...,v_{n})}^{t}$ of length $n$ such that $(v_{1},...,v_{n}) \in Um_{n}(R)$. We denote the set of unimodular columns of length $n$ over $R$ by $Um_{n}^{t}(R)$. The group $GL_{n}(R)$ and hence all its subgroups act on the left on $Um_{n}^{t}(R)$ by matrix multiplication. If $\chi$ is an alternating invertible matrix of rank $2n$, then the group $Sp(\chi)$ acts on the left on $Um_{2n}^{t}(R)$. For integers $n,i$ with $1 \leq i \leq n$, we let $e_{i,n} = \pi_{i,n}^{t}$ and $e_{n}=e_{n,n}$.\\
For any integer $n \geq 1$, we let $A_{2n}(R)$ denote the set of invertible alternating matrices of rank $2n$. We denote by $\psi_{2n} \in  A_{2n} (R)$ the matrix which is defined inductively by
 
\begin{center}
$\psi_2 =
\begin{pmatrix}
0 & 1 \\
- 1 & 0
\end{pmatrix}
$
\end{center}
 
\noindent and $\psi_{2n+2} = \psi_{2n} \perp \psi_2$. Then we obtain embeddings $A_{2m} (R) \rightarrow A_{2n} (R)$, $M \mapsto M \perp \psi_{2n-2m}$ for any $m < n$ and we let $A (R)$ denote the direct limit of the sets $A_{2n} (R)$ under these embeddings. If $M \in A_{2m} (R)$ and $N \in A_{2n} (R)$, we call them equivalent, $M \sim N$, if there is an integer $l \geq 1$ and a matrix $E \in E_{2n+2m+2l}(R)$ such that

\begin{center}
$M \perp \psi_{2n+2l} = E^{t} (N \perp \psi_{2m+2l}) E$.
\end{center}

It is easy to see that this defines an equivalence relation on $A(R)$ and the corresponding set of equivalence classes $A(R)/{\sim}$ is denoted $W'_E (R)$. It follows from \cite[\S 3]{SV} that the orthogonal sum of matrices induces the structure of an abelian group on $W'_E (R)$. The subgroup generated by alternating invertible matrices with Pfaffian $1$ will be denoted $W_E (R)$ and is called the elementary symplectic Witt group of $R$.\\
Following \cite[Section 2.1]{Sy2}, we have an exact Karoubi periodicity sequence

\begin{center}
$K_{1}{Sp} (R) \xrightarrow{f} K_{1} (R) \xrightarrow{H} W'_{E} (R) \xrightarrow{\eta} K_{0}{Sp} (R) \xrightarrow{f'} K_{0} (R)$
\end{center}

involving the group $W'_E (R)$ and classical algebraic $K$-theory and symplectic $K$-theory groups. Here the homomorphisms $f$ and $f'$ are both the usual forgetful homomorphisms. Furthermore, the map $K_{1} (R) \xrightarrow{H} W'_{E} (R)$ is given by the assignments $M \mapsto M^{t} \psi_{2n} M$ for all $M \in GL_{2n} (R)$, while the homomorphism $W'_{E} (R) \xrightarrow{\eta} K_{0}{Sp} (R)$ is given by the assignments $M \mapsto [R^{2n}, M] - [R^{2n},\psi_{2n}]$ for all $M \in A_{2n} (R)$. Again following \cite[Section 2.1]{Sy2}, the sequence above can be rewritten as

\begin{center}
$K_{1}{Sp} (R) \xrightarrow{f} SK_{1} (R) \xrightarrow{H} W_{E} (R) \xrightarrow{\eta} K_{0}{Sp} (R) \xrightarrow{f'} K_{0} (R)$.
\end{center}

We let $W_{SL}(R)$ denote the cokernel of the homomorphism $SK_{1} (R) \xrightarrow{H} W_E (R)$. We refer the reader to \cite[Section 2.1]{Sy2} for details on the groups $W'_E (R)$, $W_E (R)$ and $W_{SL}(R)$. The following lemma characterizes matrices which lie in the kernel of the map $H$ above and will be used in the proof of Theorem \ref{T3.1}:

\begin{Lem}[{{\cite[Lemma 2.1]{Sy2}}}]\label{L2.1}
Let $R$ be a commutative ring and let $\chi \in A_{2n} (R)$. If $\varphi \in GL_{2n} (R)$ such that its class $[\varphi] \in K_{1} (R)$ lies in $\ker (H)$, then there are $m \in \mathbb{N}$ and $\varphi' \in SL_{2n+2m} (R)$ such that $[\varphi] = [\varphi'] \in K_1 (R)$ and $\varphi'$ is symplectic with respect to $\chi \perp \psi_{2m}$.
\end{Lem}

The following lemma is another important ingredient for the proof of Theorem \ref{T3.1}:

\begin{Lem}[{{\cite[Lemma 3.5]{Sy2}}}]\label{L2.2}
Let $R$ be a commutative ring and let $\chi_{1}$ and $\chi_{2}$ be invertible alternating matrices of rank $2n$ over $R$ such that ${\varphi}^{t} (\chi_{1} \perp \psi_{2}) \varphi = \chi_{2} \perp \psi_{2}$ holds for some $\varphi \in SL_{2n+2}(R)$. Furthermore, let $\chi = \chi_{1} \perp \psi_{2}$. If the equality $Um^{t}_{2n+2}(R) = (E_{2n+2}(R) \cap {Sp} (\chi)) e_{2n+2}$ holds, then one has ${\psi}^{t} \chi_{2} \psi = \chi_{1}$ for some $\psi \in SL_{2n}(R)$ such that $[\psi] = [\varphi] \in K_{1}(R)$.
\end{Lem}

Now let $n \geq 1$ and let $v = (v_{1},...,v_{n}) \in Um_{n}(R)$ be a unimodular row and $w={(w_{1},...,w_{n})}^{t} \in Um^{t}_{n}(R)$ a unimodular column such that $\sum_{i=1}^{n} v_{i} w_{i} = 1$ (i.e., $w$ defines a section of $v$). Then Suslin defined matrices $\alpha_{n} (v,w) \in SL_{2^{n-1}} (R)$ called Suslin matrices in \cite{S2}; he then showed that for any such $v$ with section $w$ there exists an invertible $n \times n$-matrix $\beta (v,w)$ whose first row is $(v_{1},...,v_{n}^{(n-1)!})$ such that the classes of $\beta (v,w)$ and $\alpha_{n} (v,w)$ coincide in $K_{1} (R)$ (cf. \cite[Proposition 2.2 and Corollary 2.5]{S3}).

\begin{Lem}[{{\cite[Lemma 2.2]{Sy2}}}]\label{L2.3}
Let $R$ be an affine algebra over a perfect field $k$ and let $l \geq 1$ and $n \geq 3$. Furthermore, let $v = (v_{1},...,v_{n}) \in Um_{n} (R)$ be a unimodular row with a section $w$ and let $v' = (v_{1},...,v_{n}^{l}) \in Um_{n} (R)$ with any section $w'$. If $l$ is even, then $[\alpha_{n}(v',w')] \in SK_{1} (R)$ is an $\dfrac{l}{2}$-fold multiple of an element in $SK_{1}(R)$.
\end{Lem}

For all $n \geq 1$ and any field $k$, let $S_{2n-1} = (k[x_{1},...,x_{n},y_{1},...,y_{n}]/\langle \sum_{i=1}^{n} x_{i}y_{i} - 1 \rangle$. Note that $x=(x_{1},...,x_{n})$ is a unimodular row of length $n$ over $S_{2n-1}$ with section given by $y=(y_{1},...,y_{n})$. Now assume that $R$ is an affine $k$-algebra. Then it is easy to see that one has a correspondence

\begin{center}
$\{(a,b)|a,b \in \mathit{Um}_{n} (R), a b^{t} = 1\} = Hom_{{k}\textit{-Alg}} (S_{2n-1},R)$,
\end{center}

where $k$-Alg is the category of $k$-algebras. The rings $S_{2n-1}$ will be considered in the proof of Theorem \ref{T3.1}.\\
We now conclude this section by recalling two important results from \cite{B} on certain commutative rings satisfying the condition $\textbf{P}$ from the introduction.

\begin{Thm}[{{\cite[Theorem 5.1]{B}}}]\label{T2.4}
Let $R$ be a regular integral domain satisfying condition $\textbf{P}$ from Theorem \ref{T1.1}. Then the map $SL_{d+1}(R)/E_{d+1}(R) \rightarrow SK_{1}(R)$ is bijective.
\end{Thm}

\begin{Thm}[{{\cite[Theorem 4.1, Corollary 4.3]{B}}}]\label{T2.5}
Let $R$ be a ring satisfying condition $\textbf{P}$ from Theorem \ref{T1.1}. Then, for any unimodular row $v \in Um_{d+1}(R)$ and integer $n \geq 1$, there are $w = (w_{1},...,w_{d+1}) \in Um_{d+1}(R)$ and $\varphi \in E_{d+1}(R)$ such that $v \varphi = (w_{1},...,w_{d+1}^{n})$.
\end{Thm}

\section{Results}\label{Results}\label{3}

\begin{Thm}\label{T3.1}
Let $R$ be a regular integral domain satisfying condition $\textbf{P}$ from Theorem \ref{T1.1}. Let $\chi \in A_{d+1}(R)$ be an alternating invertible matrix of rank $d+1$. Then $Sp (\chi)$ acts transitively on $Um_{d+1}(R)$.
\end{Thm}

\begin{proof}
First we realize that it suffices to show the following statement: If $v = (v_{1},...,v_{d+1}) \in Um_{d+1}(R)$, then there exists $\varphi \in SL_{d+1}(R)$ with first row $v$ whose class $[\varphi] \in K_{1} (R)$ lies in the image of $K_{1}{Sp} (R) \xrightarrow{f} K_{1} (R)$.\\
Indeed, assume that there is $\varphi$ as above; then Lemma \ref{L2.1} shows that, for some $m \geq 0$, there is $\psi \in Sp(\chi \perp \psi_{2m})$ such that $[\psi]=[\varphi] \in K_1 (R)$. As a matter of fact, we can actually assume that $m = 0$ by Lemma \ref{L2.2} because $E_{d+1+2n}(R) \cap Sp(\chi \perp \psi_{2n})$ acts transitively on $Um_{d+1+2n} (R)$ for $n \geq 1$ by \cite[Chapter IV, Theorem 3.4]{HB} and \cite[Lemma 5.5]{SV}. By Theorem \ref{T2.4}, the map $SL_{d+1} (R)/E_{d+1} (R) \rightarrow SK_{1} (R)$ is injective. Therefore $\varphi {\psi}^{-1} \in E_{d+1} (R)$. By \cite[Lemma 5.5]{SV} the equality $\pi_{1,d+1} E_{d+1} (R) = \pi_{1,d+1} ({E}_{d+1} (R) \cap Sp(\chi))$ holds, so there exists $\psi' \in {E}_{d+1} (R) \cap Sp (\chi)$ such that $\pi_{1,d+1}\varphi {\psi}^{-1} = \pi_{1,d+1} \psi'$. As a consequence, $v = \pi_{1,d+1} \varphi = \pi_{1,d+1} \psi' \psi$ is indeed the first row of a matrix in $Sp (\chi)$.\\
So let us now show that for $v = (v_{1},...,v_{d+1}) \in Um_{d+1}(R)$ there exists $\varphi \in SL_{d+1}(R)$ with first row $v$ whose class $[\varphi] \in K_{1} (R)$ lies in the image of $f: K_{1}{Sp} (R) \rightarrow K_{1} (R)$ or, equivalently, in the kernel of $H: SK_{1}(R) \rightarrow W_E (R)$. First of all, we notice that by Theorem \ref{T2.5} any unimodular row of length $d+1$ over $R$ can actually be transformed via elementary matrices to a row of the form $v = (v_{1},...,v_{d+1}^{2{d!}^{2}})$. Therefore we only need to consider rows of the form $v = (v_{1},...,v_{d+1}^{2{d!}^{2}})$.\\
For this purpose, we also consider the row $v' = (v_{1},...,v_{d+1}^{2d!})$ with a chosen section $w'$; then the first row of the matrix $\beta_{d+1}(v',w') \in SL_{d+1}(R)$ is precisely the row $v$. It is well-known that $[\beta_{d+1}(v',w')]$ lies in the image of the forgetful map $K_{1}{Sp} (R) \xrightarrow{f} K_{1} (R)$ if $d+1 \equiv 2~mod~4$ (cf. \cite[Proposition 3.3.3]{AF}), which finishes the proof of the theorem in this case.\\
It remains to prove the theorem for $d+1 \equiv 0~mod~4$. In this case, by \cite[Proposition 2.7]{Sy1}, we have that $W_E (S_{2(d+1)-1}) = \mathbb{Z}/2\mathbb{Z}$ and therefore $W_E (S_{2(d+1)-1})$ is $2$-torsion. If we now let $x'=(x_{1},...,x_{d},x_{d+1}^{2d!}) \in Um_{d+1}(S_{2(d+1)-1})$ with a chosen section $y'$, then Lemma \ref{L2.3} shows that $[\beta_{d+1} (x',y')] \in SK_{1}(S_{2(d+1)-1})$ is a $d!$-fold multiple of an element in $SK_{1}(S_{2(d+1)-1})$. Consequently, we must have $H([\beta_{d+1}(x',y')]) = 0$ as $W_E (S_{2(d+1)-1})$ is $2$-torsion.\\
This implies that $H([\beta_{d+1}(v',w')]) = 0$ as follows: We let $v''=(v_{1},...,v_{d+1}) \in Um_{d+1}(R)$ with some chosen section $w''$. We obtain an induced ring homomorphism $\varphi: S_{2(d+1)-1} \rightarrow R, (x,y) \mapsto (v'',w'')$ fitting into a commutative diagram
\begin{center}
$\begin{xy}
  \xymatrix{
      SK_1 (S_{2(d+1)-1}) \ar[r]^{H} \ar[d]_{\varphi_{\ast}}    &   W_E (S_{2(d+1)-1}) \ar[d]^{\varphi_{\ast}}  \\
      SK_{1} (R) \ar[r]_{H}             &   W_E (R).   
  }
\end{xy}$
\end{center}
Since $\varphi_{\ast}: SK_{1} (S_{2(d+1)-1}) \rightarrow SK_{1}(R)$ sends $[\beta_{d+1} (x',y')]$ to $[\beta_{d+1} (v',w')]$, the commutative diagram implies that indeed $H([\beta_{d+1}(v',w')]) = 0 \in W_E (R)$. This finishes the proof of the theorem in case $d+1 \equiv 0~mod~4$.
\end{proof}

\begin{Rem}
Theorem \ref{T3.1} is essentially a consequence of the results obtained in \cite{B} and \cite{Sy2}. It was already observed in \cite[Remark 5.4]{B} that the proof of \cite[Theorem 3.6]{Sy2} together with Theorem \ref{T2.4} and Theorem \ref{T2.5} imply that $Sp_{d+1}(R)$ acts transitively on $Um_{d+1}(R)$. Lemma \ref{L2.1} and Lemma \ref{L2.2} enable us to prove a more general statement in Theorem \ref{T3.1}. For the convenience of the reader, we have given a full proof of this statement.
\end{Rem}

It follows from \cite[Theorem 3.1]{Sy1} that there is a well-defined map

\begin{center}
$V_{\theta_{0}}: Um_3 (R)/SL_3 (R) \rightarrow \tilde{V}_{SL}(R)$
\end{center}

called Vaserstein symbol modulo SL associated to any fixed isomorphism $\theta_{0}: \det(R^2) \xrightarrow{\cong} R$, where $\tilde{V}_{SL}(R)$ is an abelian group introduced in \cite[Section 2.C]{Sy1}; as explained in \cite[Section 2.C]{Sy1}, the abelian group $\tilde{V}_{SL}(R)$ is canonically isomorphic to the group $W_{SL}(R)$ from Section \ref{2}.

\begin{Thm}\label{T3.3}\label{T3.3}
Let $R$ be a regular integral domain of dimension $d=3$ satisfying condition $\textbf{P}$ from Theorem \ref{T1.1}. Let $\theta_{0}: R \xrightarrow{\cong} \det (R^2)$ be a fixed isomorphism. Then $V_{\theta_{0}}: Um_{3}(R)/SL_{3}(R) \xrightarrow{\cong} \tilde{V}_{SL}(R)$ is a bijection.
\end{Thm}

\begin{proof}
It follows from \cite[Theorems 3.2 and 3.6]{Sy1} that the generalized Vaserstein symbol modulo SL is bijective if $SL_{5}(R)$ acts transitively on $Um_{5}(R)$ and $Sp (\chi)$ acts transitively on $Um_{4}^{t}(R)$ for any invertible alternating $4\times4$-matrix $\chi$ over $R$. It follows directly from \cite[Chapter IV, Theorem 3.4]{HB} that $SL_{5}(R)$ acts transitively on $Um_{5}(R)$. Theorem \ref{T3.1} shows that $Sp (\chi^{-1})$ acts transitively on the right on $Um_{4}(R)$ for any invertible alternating $4 \times 4$-matrix $\chi$; the fact that $Sp(\chi)$ acts transitively on the left on $Um_{4}^{t}(R)$ then follows by transposition.
\end{proof}

\begin{Thm}\label{T3.4}\label{T3.4}
Let $R$ be a regular integral domain of dimension $d=3$ satisfying condition $\textbf{P}$ from Theorem \ref{T1.1}. Then all stably free $R$-modules of rank $2$ are free if and only if $W_{SL}(R) = 0$.
\end{Thm}

\begin{proof}
Assume $P$ is a stably free $R$-module of rank $2$ and $P \oplus R^n \cong R^{n+2}$ for some integer $n \geq 1$. Then, by \cite[Theorem 2.8]{B}, we obtain an isomorphism $P \oplus R \cong R^3$. In particular, $P$ is defined via a unimodular row $v$ of length $3$ over $R$ and $P$ is free if and only if there is $\varphi \in SL_{3}(R)$ with first row $v$. The theorem now follows from Theorem \ref{T3.3} and the fact that $\tilde{V}_{SL}(R)$ is isomorphic to $W_{SL}(R)$.
\end{proof}

\begin{Rem}\label{R3.5}
Theorem \ref{T3.3} actually implies a slightly stronger statement than the statement of Theorem \ref{T3.4}: It follows from the discussion in \cite[Section 2.A]{Sy1} that the set $Um_{3}(R)/SL_{3}(R)$ in Theorem \ref{T3.3} corresponds precisely to the set of isomorphism classes of oriented stably free $R$-modules of rank $2$. This set turns out to be in bijection with the abelian group $W_{SL}(R)$ by Theorem \ref{T3.3} and therefore inherits an abelian group structure.
\end{Rem}


\begin{thebibliography}{xxxxxx}
\bibitem[AF]{AF} A. Asok, J. Fasel, An explicit $KO$-degree map and applications, J. Topology 10 (2017), 268-300
\bibitem[B]{B} S. Banerjee, Zero cycles, Mennicke symbols and $\mathrm{K}_1$-stability of certain real affine algebras Int. Math. Res. Not. IMRN, vol. 2025, Issue 15, August 2025
\bibitem[FRS]{FRS} J. Fasel, R. A. Rao and R. G. Swan, On stably free modules over affine algebras, Publ. Math. Inst. Hautes \'Etudes Sci. 116 (2012), 223-243
\bibitem[HB]{HB} H. Bass, Algebraic K-theory, Benjamin, New York, 1968
\bibitem[S1]{S1} A. A. Suslin, A cancellation theorem for projective modules over algebras, Dokl. Akad. Nauk SSSR 236 (1977), 808-811
\bibitem[S2]{S2} A. A. Suslin, On stably free modules, Math. U.S.S.R. Sbornik (1977), 479-491
\bibitem[S3]{S3} A. A. Suslin, Mennicke symbols and their applications in the $K$-theory of fields, in Algebraic $K$-theory, Part I (Oberwolfach, 1980), Lecture Notes in Math., vol. 966, 334-356, Springer, Berlin-New York, 1982
\bibitem[SV]{SV} A. A. Suslin, L.N. Vaserstein, Serre's problem on projective modules over polynomial rings, and algebraic K-theory, Izv. Akad. Nauk. SSSR Ser. Mat. 40 (1976), 993-1054
\bibitem[Sy1]{Sy1} T. Syed, The cancellation of projective modules of rank 2 with a trivial determinant, Algebra \& Number Theory 15 (2021), no. 1, 109-140
\bibitem[Sy2]{Sy2} T. Syed, Symplectic orbits of unimodular rows, Journal of Algebra (2024), Volume 646, 478-493
\end{thebibliography}
\end{document}